\documentclass[12pt]{article}

\usepackage{amsthm}
\usepackage{amsmath}
\usepackage{amssymb}
\usepackage{graphicx}
\usepackage{float}
\usepackage[makeroom]{cancel}

\graphicspath{{C:/Users/Alex/Documents/Maths/}}

\begin{document}
	
\title{Palindromic Density}
\author{Alex Burlton}

\maketitle

\newcommand{\N}{n}
\newcommand{\halfN}{\frac{\N{}}{2}}
\newcommand{\B}{b}
\newcommand{\Xbn}{X_{\B{}}^\N{}}
\newcommand{\Pbn}{P_{\B{}}^\N{}}
\newcommand{\Xbnsize}{{{\N{} + \B{} - 1} \choose {\B{} - 1}}}

\newtheorem{theorem}{Theorem}[section]
\newtheorem{corollary}{Corollary}[theorem]
\theoremstyle{definition}
\newtheorem{definition}{Definition}[section]
\theoremstyle{remark}
\newtheorem*{remark}{Remark}

\begin{abstract}
\noindent In this paper we consider the palindromes that can be formed by taking unordered sets of $\N{}$ elements from an alphabet of $\B{}$ letters. In particular, we seek to find the probability that given a random member of this space we are able to re-arrange its elements to form a palindrome. We conclude by exploring the behaviour of this probability as $\N{}, \B{} \to \infty$. 
\end{abstract}

\section{Motivation}

We begin with a short example to illustrate the problem. Suppose we have the set \{0, 1, ..., 9\}. What unordered words of length 5 can we make? By unordered, we simply mean that `01156' is the same word as `10651'. If we were to pick one of these words at random, what's the probability that we could re-arrange the digits to create a palindrome? 
\\\\
Note that we assume every word to have equal probability of being picked, so this is $not$ the same as asking ``If we pick 5 random digits from 0-9, what are the chances we could re-arrange them to create a palindrome?". In the second phrasing of the question, `01156' is more likely than `00000' as there are more sequences of `picks' that will result in it.
\\\\
The more intriguing problem arises when we generalise this, which we do in two ways. Rather than the digits 0-9, we instead have a set of $\B{}$ distinct characters, and rather than words of length 5 we will consider words of length $\N{}$. In order for the problem to make sense, we will require $\N{}, \B{} > 1$.

\section{Definitions and Notation}

Before we start tackling this problem, we must first lay out some groundwork to formalise what we mean by this probability.
\\
\begin{definition}[Multiset]
	A multiset is a set that allows repeated elements, for example \{1, 1, 2\}. In particular, and unlike regular sets, we have that $\{1, 1, 2\} \neq \{1, 2\}$. More formally, a multiset can be thought of as a set, $A$, paired with a counting function $f: A \rightarrow \mathbb{N}$.
\end{definition}
\mbox{}
\begin{definition}[Multiset Space]\label{Xbn}
Let $X_\B{}$ be a set of $\B{}$ distinct elements, where $\B{} > 1$. Then for $\N{} > 1$ we define the \textit{multiset space} of cardinality $\N{}$ on $X_\B{}$ as:

\[\Xbn{} := \{X = \{x_1, ..., x_\N{}\} \:|\: x_j \in X_{\B{}},\: X \text{ is a multiset}\}\]
\end{definition}
\mbox{}
\begin{definition}[Palindromic]\label{palindromic}
$X \in \Xbn{}$ is called palindromic if:

\[\exists \: i_1, ... i_\N{} \in \{1, ..., \N{}\} \quad i_k \ne i_l \quad \forall k \ne l\]
\[\text{s.t.} \quad (x_{i_1}, ... , x_{i_\N{}}) = (x_{i_\N{}}, ..., x_{i_1})\]
\end{definition}
\mbox{}
\begin{remark} 
	In simple terms, this means that the elements of the set can be arranged into a word that reads the same forwards or backwards. For example, $\{1, 1, 2, 2, 3\}$ is palindromic (write it as 12321), whilst $\{1, 1, 2, 3, 4\}$ is not.
\end{remark}
\mbox{}
\begin{definition} [Palindomic Density]
	Define $\Pbn{} \subset \Xbn{}$ as:
	
	\[\Pbn{} := \{X \in \Xbn{} \:|\: X \: \text{is palindromic}\}\]
	
	\mbox{}\\\noindent Then we define the \textit{palindromic density} of $\Xbn{}$ as:
	
	\[PD(\Xbn{}) := \frac{|\Pbn{}|}{|\Xbn{}|}\]
\end{definition} 
\mbox{}\\
\noindent In this paper we find a closed expression for $PD(\Xbn{})$ in terms of $\N{}$ and $\B{}$.
\pagebreak
\section{Solution for $\N{} = 5$, $\B{} = 10$}

With our problem precisely formulated, we can return to our initial example where $\N{} = 5$ and $\B{} = 10$. It is difficult to compute the size of $X_{10}^5$ due to the differing numbers of repetitions that can occur within its elements. A way of breaking down this problem is to partition $X_{10}^5$ as follows:

\[X_{10}^5 = \mathop{\dot{\bigcup}}_{1 \leq i \leq 7} Y_{i}\] where:
\begin{equation*}
\begin{split}
Y_1 & := \{\{a, b, c, d, e\} \in X_{10}^5 \:|\: a \ne b \ne c \ne d \ne e \}
\\
Y_2 & := \{\{a, a, b, c, d\} \in X_{10}^5 \:|\: a \ne b \ne c \ne d \}
\\
Y_3 & := \{\{a, a, b, b, c\} \in X_{10}^5 \:|\: a \ne b \ne c\}
\\
Y_4 & := \{\{a, a, a, b, c\} \in X_{10}^5 \:|\: a \ne b \ne c\}
\\
Y_5 & := \{\{a, a, a, b, b\} \in X_{10}^5 \:|\: a \ne b\}
\\
Y_6 & := \{\{a, a, a, a, b\} \in X_{10}^5 \:|\: a \ne b\}
\\
Y_7 & := \{\{a, a, a, a, a\} \in X_{10}^5\}
\end{split}
\end{equation*}

\mbox{}\\In our partition, the palindromic sets are $Y_3$, $Y_5$, $Y_6$ and $Y_7$. Computing the sizes of these sets is simple - for example, $|Y_2| = \frac{10 \times 9 \times 8 \times 7}{3!}  = 840$. The first set is a slightly special case: $|Y_1| = {10 \choose 5} = 252$.
\\\\Thus:

\[PD(X_{10}^5)\ = \frac{|P_{10}^5|}{|X_{10}^5|}\ = \frac{|Y_3| + |Y_5| + |Y_6| + |Y_7|}{|Y_1| + |Y_2| + |Y_3| + |Y_4| + |Y_5| + |Y_6| + |Y_7|} = \frac{550}{2002}\]

\pagebreak
\section{The General Solution}
We start by finding $|\Xbn|$, using an existing combinatorial result, with a proof adapted from \cite{adiscreteintroduction}.
\begin{theorem}
	\[|\Xbn{}| = { {\N{} + \B{} - 1} \choose {\B{} - 1}} \]
\end{theorem}
	
\begin{proof}
	Since $\B{}$ is finite, we can define an ordering on $X_\B{}$ and represent elements of $\Xbn{}$ as $\{x_1, x_2, ..., x_\N{}\} \text{ where } x_i \preceq x_j \:\: \forall i < j$.
	
	\mbox{}\\Every element of $X \in \Xbn{}$ can then be denoted uniquely as a series of `stars and bars': we can lay out $\N{}$ `stars' to represent the elements of the ordered multiset $X$, and then place $(\B{} - 1)$ bars among them to separate out the distinct characters.
	
	\mbox{}\\For example, the multiset $\{0, 0, 1, 3, 4\} \in X_{10}^{5}$ would be represented as follows:
	
	\[**|*||*|*|||||\]
	
	\mbox{}\\For every such arrangement, it is possible to get back to an element of $\Xbn{}$, and so our problem is reduced to finding the number of ways to arrange $\N{}$ stars and $(\B{} - 1)$ bars in this way. Note that there are $(\N{} + \B{} - 1)$ positions in total, and we must choose  $(\B{} - 1)$ of these for our bars. We therefore have that the number of such representations is ${ {\N{} + \B{} - 1} \choose {\B{} - 1}}$, as required.
\end{proof}

\begin{theorem}\label{evenproof}
	If n is even, then
	\[|\Pbn{}| = | X_{\B{}}^{\halfN{}} |\]
\end{theorem}

\begin{proof}
	We seek to construct a bijective function $f : X_{\B{}}^{\halfN{}} \to \Pbn{}$. In order for this to be well-defined, we will use the ordering of $X_\B{}$ discussed in the previous proof to represent elements in the domain uniquely.
	
	\mbox{}\\ \noindent Now construct a doubling function $f : X_{\B{}}^{\halfN{}} \to \Xbn{}$ \:in the following way:
	
	\[f(\{x_1, x_2, ... , x_\halfN{}\}) := \{x_1, x_1, x_2, x_2, ... , x_\halfN{}, x_\halfN{}\}\]
	\mbox{}\\\\
	Clearly $f$ is injective. We can also verify that all sets produced by $f$ are palindromic by arranging the elements as $(x_1, x_2, ... , x_\halfN, x_\halfN, ..., x_2, x_1).$ So  $f : X_{\B{}}^{\halfN{}} \to \Pbn{}$. It remains to show that $f$ is surjective.
	\\\\Suppose $Y \in \Pbn{}$. Then by definition \ref{palindromic}:
	
	\[\exists \: i_1, ... i_\N{} \in \{1, ..., \N{}\} \quad i_k \ne i_l \quad \forall k \ne l\]
	\[s.t. \quad (y_{i_1}, ... , y_{i_\halfN{}}, y_{i_{\halfN{} + 1}} , ... y_{i_\N{}}) = (y_{i_\N{}}, ... , y_{i_{\halfN{} + 1}}, y_{i_\halfN{}} , ... y_{i_1})\]
	\mbox{}\\
	In particular, this means we can write $Y$ as $\{y_{i_1}, y_{i_1}, y_{i_2}, y_{i_2}, ... , y_{i_\halfN{}}, y_{i_\halfN{}}\}$, so $f(X) = Y$ where $X := \{y_{i_1}, y_{i_2}, ... , y_{i_\halfN{}}\} \in X_{\B{}}^{\halfN{}}$
	\mbox{}\\\\
	Hence f is bijective and the statement holds.
\end{proof}

\begin{theorem}
	If n is odd, then
	
	\[|\Pbn{}| = b|P_{\B{}}^{\N{}-1}| \]
\end{theorem}
\begin{proof}
	From the proof of \ref{evenproof}, we know that every element of $P_\B{}^{\N{} - 1}$ can be written as $\{x_1, x_1, x_2, x_2, ..., x_{\frac{\N{}-1}{2}}, x_{\frac{\N{}-1}{2}}\}$. It is clear from this expression that we cannot change a single element of the set and preserve the palindromic property. We can therefore deduce that every element of $P_\B{}^{\N{} - 1}$ differs by at least two elements. 
	\\
	
	\noindent Suppose $X \in P_\B{}^{\N{} - 1}$. Then if we add another element to it, $x_* \in X_\B{}$, the resulting set must belong to $\Pbn{}$. This holds because we can arrange our new set as follows, showing that it is palindromic:
	
	\[(x_1, x_2, ... x_{\frac{\N{}-1}{2}}, x_*, x_{\frac{\N{}-1}{2}}, ..., x_2, x_1)\]
	
	\mbox{}\\
	\noindent For every $X \in P_\B{}^{\N{} - 1}$, it is possible to produce $\B{}$ elements of $\Pbn$ in this way. These elements are uniquely determined by our choice of $X$ and $x_*$ - recall that elements of $P_\B{}^{\N{} - 1}$ differ by at least two elements. Hence
	
	\[ |\Pbn{}| \geq \B{}|P_\B{}^{\N{} - 1}| \]
	
	\mbox{}\\
	\noindent We now show that every element of $\Pbn{}$ can be made in this way. Suppose $Y \in \Pbn{}$. Then by definition \ref{palindromic}:
	
	\[\exists \: i_1, ... i_\N{} \in \{1, ..., \N{}\} \quad i_k \ne i_l \quad \forall k \ne l\]
	\[s.t. \quad (y_{i_1}, ... , y_{i_\frac{\N{} - 1}{2}}, y_{i_\frac{\N{} + 1}{2}}, y_{i_\frac{\N{} + 3}{2}} , ... y_{i_\N{}}) = (y_{i_\N{}}, ... , y_{i_\frac{\N{} + 3}{2}}, y_{i_\frac{\N{} + 1}{2}}, y_{i_\frac{\N{} - 1}{2}} , ... y_{i_1})\]
	\mbox{}\\
	In particular, this means we can write $Y$ as $\{y_{i_1}, y_{i_1}, , ... , y_{i_\frac{\N{} - 1}{2}}, y_{i_\frac{\N{} - 1}{2}}, y_{i_\frac{\N{} + 1}{2}}\}$. But this is exactly what we get by adding the element $x_* := y_{i_\frac{\N{} + 1}{2}}$ to $X := \{y_{i_1}, y_{i_1}, , ... , y_{i_\frac{\N{} - 1}{2}}, y_{i_\frac{\N{} - 1}{2}}\} \in P_\B{}^{\N{} - 1}$.
	\mbox{}\\\\
	Hence $|\Pbn{}| \leq \B{}|P_\B{}^{\N{} - 1}|$ and the statement holds.
\end{proof}

\mbox{}\\
\noindent We now combine the above three theorems to arrive at the closed form we were looking for.
\mbox{}\\
\begin{corollary}\label{PDresult}
	\[PD(\Xbn{}) = 
	\begin{cases} 
	\hfill \frac{{ {\halfN{} + \B{} - 1} \choose {\B{} - 1}}}{\Xbnsize{}}    \hfill & \text{ if $\N{}$ is even} \\\\
	\hfill \frac{b{{\frac{n-1}{2} + \B{} - 1} \choose {\B{} - 1}}}{\Xbnsize{}} \hfill & \text{ if $\N{}$ is odd} \\
	\end{cases}\]
\end{corollary}

\section{Behaviour as $\B{}, \N{} \rightarrow \infty$} \label{limsection}

The result obtained in the previous section can be expanded and simplified into the following form, which has advantages when exploring convergence:

\[PD(\Xbn{}) = 
\begin{cases} 
\hfill  \prod\limits_{i = \frac{\N{} + 2}{2}}^{\N{}} \frac{i}{i+\B{}-1} \hfill & \text{ if $\N{}$ is even} \\\\
\hfill  \B{} \prod\limits_{i = \frac{\N{} + 1}{2}}^{\N{}} \frac{i}{i+\B{}-1} \hfill & \text{ if $\N{}$ is odd} \\
\end{cases}\]

\noindent First, we consider the case when $\B{} \to \infty$. It makes sense that if we increase $\B{}$, the probability of a set being palindromic should decrease. This is because with more distinct elements to choose from, a smaller proportion of words (or multisets) will contain the number of repetitions we need. The next result, then, shouldn't come as much of a surprise.

\begin{theorem}
	\[PD(\Xbn{}) \rightarrow 0 \quad\text{as}\quad \B{} \rightarrow \infty\]
\end{theorem}
\begin{proof}
	We notice that $i > 1$ for every term in the product. Hence, if $\N{}$ is even:
	\begin{equation*}
	\begin{split}
		\prod\limits_{i = \frac{\N{} + 2}{2}}^{\N{}} \frac{i}{i+\B{}-1} &< \frac{\N{}}{\N{} + \B{} - 1}  \: \rightarrow 0 \quad \text{as} \quad \B{} \rightarrow \infty
	\end{split}
	\end{equation*}
	If $\N{}$ is odd, then we must use a slightly different bound:
	\begin{equation*}
	\begin{split}
	\B{}\prod\limits_{i = \frac{\N{} + 1}{2}}^{\N{}} \frac{i}{i+\B{}-1} & < \B{}\prod\limits_{i = \frac{\N{} + 1}{2}}^{\N{}} \frac{i}{\B{}} \quad\quad \text{as} \quad i, \B{} > 1 \quad \forall i
	\\[1ex]& < \B{}\prod\limits_{i = \frac{\N{} + 1}{2}}^{\N{}} \frac{n}{\B{}} \quad\quad \text{as} \quad \N{} \geq i \quad \forall i
	\\[1ex]& = \frac{\B{}\N{}^{\frac{\N{} + 1}{2}}}{b^{\frac{\N{} + 1}{2}}} 
	\\[1.5ex]& = \frac{\N{}^{\frac{\N{} + 1}{2}}}{b^{\frac{\N{} - 1}{2}}} \: \rightarrow 0 \quad \text{as} \quad \B{} \rightarrow \infty
	\end{split}
	\end{equation*}
\end{proof}

\noindent Now consider the case when $\N{} \to \infty$. At first glance, it feels intuitive that the density should go to zero again -- by increasing the length of our words we're requiring more and more letters to pair off. However, notice that once the value of $\N{}$ exceeds that of $\B{}$, we are $guaranteed$ to start encountering duplicates - things are not quite as simple here. It is at least true that the density decreases with $\N{}$, which we prove next.\\

\begin{theorem}
	Let $k \in \mathbb{N}$. Then the following results hold:
	\\
	\[\forall \B{} > 1, \: PD(X_b^{2k}) \text{ is strictly decreasing in k }\]
	\[\forall \B{} > 2, \: PD(X_b^{2k+1}) \text{ is strictly decreasing in k }\]
	
\end{theorem}
\mbox{}
\begin{remark}
	Notice that we specify $\B{} > 2$ for the odd case. Why is this? When $\B{} = 2$ and $\N{}$ is odd, we have to make an odd number of picks from \{0, 1\} to create a set of $\Xbn$. Doing this will $always$ produce a palindromic set. It's easy to verify that plugging $\B{} = 2$ into the formula from Corollary \ref{PDresult} results in a density of 1 for any odd-valued $\N{}$.
\end{remark}
\mbox{}
\begin{proof}
	We start with the even case. 
	\begin{equation*}
	\begin{split}
	PD(X_b^{2k}) & = \frac{{ {k + \B{} - 1} \choose {\B{} - 1}}}{{ {2k + \B{} - 1} \choose {\B{} - 1}}}
	\\[1ex] & = \frac{(k+\B{}-1)!}{(\B{}-1)!(k)!} \times \frac{(b-1)!(2k)!}{(2k+b-1)!}
	\\[1ex] & = \frac{(k+\B{}-1)!(2k)!}{(k)!(2k+b-1)!}
	\end{split}
	\end{equation*}
	
	\noindent Now consider the next term, which we get by incrementing $k$ by 1:
	
	\begin{equation*}
	\begin{split}
	PD(X_b^{2k+2}) & = \frac{(k+\B{})!(2k+2)!}{(k+1)!(2k+b+1)!}
	\\[1ex] & = \delta(k, \B{}) PD(X_b^{2k})
	\end{split}
	\end{equation*}
	
	\mbox{}\\where
	\begin{equation*}
	\begin{split}
	\delta(k, \B{}) & := \frac{(k+\B{})(2k+1)(2k+2)}{(k+1)(2k+\B{})(2k+\B{}+1)}
	\\[1ex] & = \frac{2(k+\B{})(2k+1)}{(2k+\B{})(2k+\B{}+1)}
	\\[1ex] & = \frac{4k^2 + (4\B{} + 2)k + 2\B{}}{4k^2 + (4\B{}+2)k + \B{}^2 + \B{}} =: \frac{\alpha(k, \B{})}{\beta(k, \B{})}
	\end{split}
	\end{equation*}
	
	\mbox{}\\Subtracting the numerator from the denominator leaves:
	
	\begin{equation*}
	\begin{split}
	\beta(k,\B{}) - \alpha(k, \B{}) & = \B{}^2 + \B{} - 2\B{}
	\\ &= \B{}(\B{} - 1) > 1 \quad \forall \B{} > 1
	\end{split}
	\end{equation*}
	
	\mbox{}\\Hence $\delta(k,\B{}) < 1 \quad \forall \B{} > 1$ and $PD(X_b^{2k})$ is strictly decreasing in $k$.
	
	\mbox{}\\For the odd case, we can expand out $PD(X^{2k+1}_\B{})$ and $PD(X^{2k+3}_\B{})$ in the same way to obtain:
	
	\[\delta(k, \B{}) = \frac{4k^2 + (4\B{} + 6)k + 6\B{}}{4k^2 + (4\B{}+6)k + \B{}^2 + 3\B{} + 2} =: \frac{\alpha(k, \B{})}{\beta(k, \B{})}\]
	
	\mbox{}\\Subtracting the numerator from the denominator again:
	
	\begin{equation*}
	\begin{split}
	\beta(k,\B{}) - \alpha(k, \B{}) & = \B{}^2 - 3\B{} + 2
	\\ &= (\B{}-2)(\B{} - 1) > 1 \quad \forall \B{} > 2
	\end{split}
	\end{equation*}
	
	\mbox{}\\Hence $\delta(k,\B{}) < 1 \quad \forall \B{} > 2$ and $PD(X_b^{2k+1})$ is strictly decreasing in $k$.
	
\end{proof}

\begin{theorem}
	\[PD(\Xbn{}) \rightarrow \begin{cases} 
	\hfill \frac{1}{2^{\B{} - 1}}    \hfill & \text{ \N{} \text{even}}, \: \N{} \rightarrow \infty \\\\
	\hfill \frac{\B{}}{2^{\B{} - 1}}    \hfill & \text{ \N{} \text{odd}\textbf{}}, \: \N{} \rightarrow \infty 
	\end{cases} \]
\end{theorem}
\begin{proof}
	We first prove the result when $\N{}$ is even. Let $\N{} = 2k, \: k \in \mathbb{N}$. Then the number of terms in our product form is $2k - (k+1) + 1 = k$. Since we're taking $k \to \infty$, we can set $k > 2b$ and expand the first and last $\B{}$ terms as follows:
	
	\begin{multline*}
	\prod\limits_{i = {k+1}}^{2k} \frac{i}{i+\B{}-1} = \frac{(k+1)}{\cancel{(k+\B{})}}\frac{(k+2)}{\cancel{(k+\B{} + 1)}} ... \frac{(k+\B{} - 1)}{\cancel{(k+2\B{} - 2)}} \frac{\cancel{(k+\B{})}}{\cancel{(k+2\B{} - 1)}} \\...  \frac{\cancel{(2k - \B{} + 1)}}{\cancel{(2k)}} \frac{\cancel{(2k - \B{} + 2)}}{(2k+1)}... \frac{\cancel{(2k-1)}}{(2k+\B{} - 2)} \frac{\cancel{(2k)}}{(2k+\B{} - 1)}
	\end{multline*}
	
	\noindent After cancelling, $(\B{} - 1)$ terms remain on the top and bottom. Thus:
	\begin{equation*}
	\begin{split}
	\prod\limits_{i = {k+1}}^{2k} \frac{i}{i+\B{}-1} & = \frac{(k+1)(k+2)...(k+\B{}-1)}{(2k+1)(2k+2)...(2k+\B{}-1)}
	\\ & = \frac{k^{\B{}-1} + F(k)}{2^{\B{}-1}k^{\B{}-1} + G(k)}
	\end{split}
	\end{equation*}
	where $F(k)$ and $G(k)$ are polynomials in $k$ of order $(\B{} - 2)$. By
	L'H\^{o}pital's Rule, it follows that:
	
	\[ \lim\limits_{k \to \infty} \bigg(\prod\limits_{i = {k+1}}^{2k} \frac{i}{i+\B{}-1}\bigg) = \frac{1}{2^{\B{}-1}} \]
	
	\mbox{}\\\noindent The same argument applies when $\N{}$ is odd, noting additionally that
	
	\[ \lim\limits_{k \to \infty} (\B{} f(k)) = \B{} \lim\limits_{k \to \infty} (f(k)) \]
\end{proof}
\pagebreak
\section{Graphing $PD(\Xbn{})$}

The product form obtained in section \ref{limsection} allows us to calculate the Palindomic Density for large values of $\N{}$ and $\B{}$ in order to produce surface graphs. These graphs demonstrate more visually the behaviour of $PD(\Xbn{})$ as $\N{}$ and $\B{}$ are increased.

\begin{figure}[H]
	\caption{$PD(\Xbn{})$ for $\N{}, \B{} \in [2, 50]$}
	\includegraphics[scale=0.5]{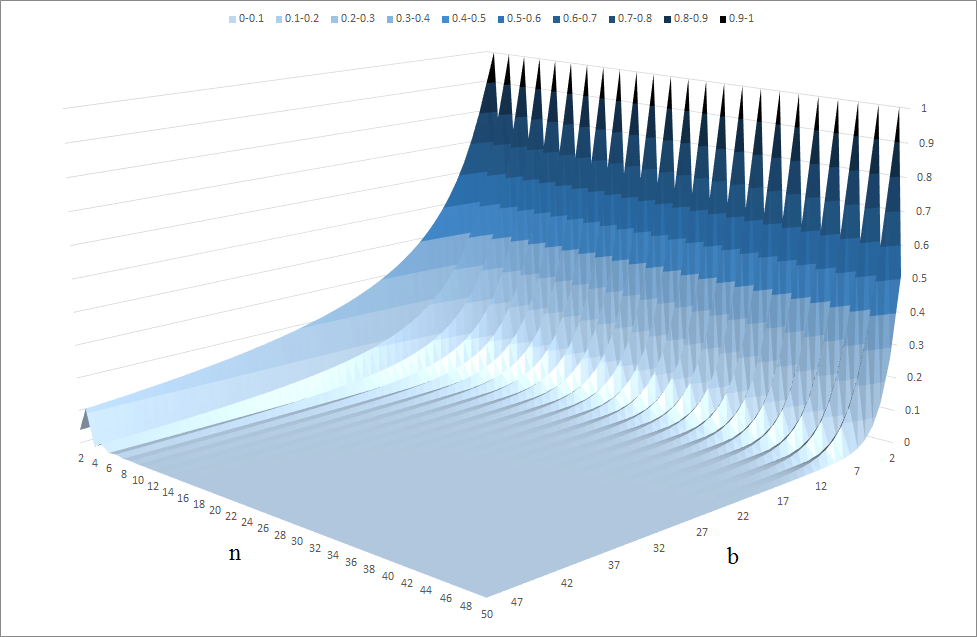}
	\centering
\end{figure}
\begin{figure}[H]
	\caption{$PD(\Xbn{})$ for $\N{}, \B{} \in [10, 50]$}
	\includegraphics[scale=0.5]{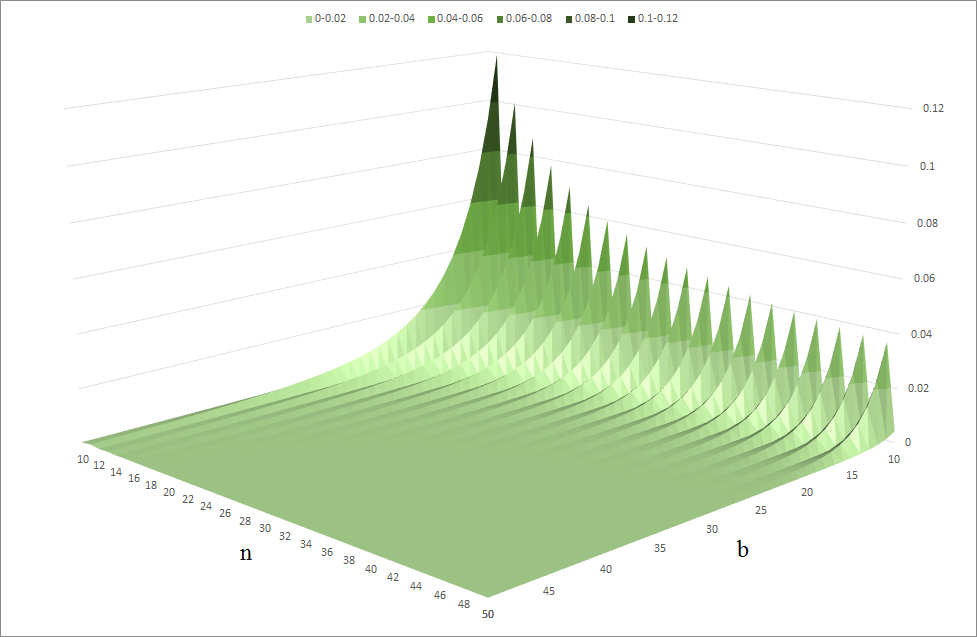}
	\centering
\end{figure}
\begin{figure}[H]
	\caption{$PD(\Xbn{})$ for $\N{}, \B{} \in [20, 50]$}
	\includegraphics[scale=0.5]{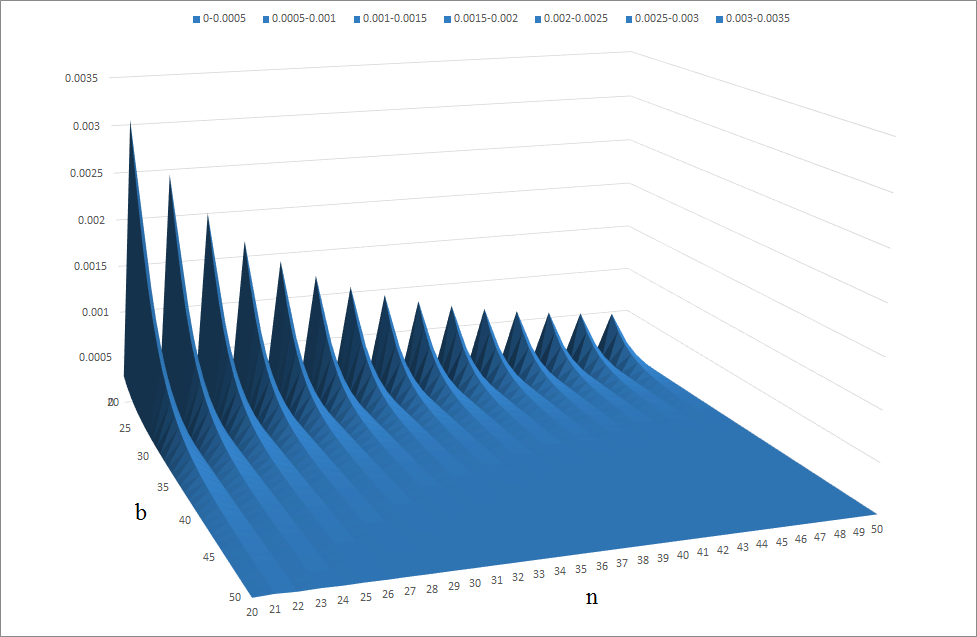}
	\centering
\end{figure}

\end{document}